\documentclass[12pt]{article}
\usepackage[UKenglish]{babel}
\usepackage[T1]{fontenc}
\usepackage{lmodern,amsmath,amsthm,amsfonts,amssymb,graphicx,float,microtype,thmtools,underscore,mathtools,thm-restate}
\usepackage[shortlabels]{enumitem}
\setlist[itemize]{topsep=0ex,itemsep=0ex,parsep=0ex}
\setlist[enumerate]{topsep=0ex,itemsep=0ex,parsep=0ex}
\usepackage[usenames,dvipsnames,svgnames,table]{xcolor}
\usepackage{todonotes}
\usepackage[unicode=true]{hyperref}
\usepackage{breakurl}
\hypersetup{ 
colorlinks,
linkcolor={blue!60!black},
citecolor={black},
urlcolor={blue!60!black},
pdftitle={Graphs Quasi-Isometric to Graphs with Bounded Treewidth}}
\usepackage[capitalise, compress, nameinlink, noabbrev]{cleveref}
\crefname{lem}{Lemma}{Lemmas}
\crefname{thm}{Theorem}{Theorems}
\crefname{ques}{Question}{Theorems}
\crefname{cor}{Corollary}{Corollaries}
\crefname{enumi}{Item}{Items}
\crefformat{equation}{(#2#1#3)}
\Crefformat{equation}{Equation #2(#1)#3}
\crefformat{enumi}{#2#1#3}
\Crefformat{enumi}{Item (#2#1#3)}

\newcommand{\defn}[1]{\textcolor{Maroon}{\emph{#1}}}
\usepackage[longnamesfirst,numbers,sort&compress]{natbib}
\makeatletter
\def\NAT@spacechar{~}
\makeatother
\usepackage[margin=28mm]{geometry}
\renewcommand{\baselinestretch}{1.1}
\allowdisplaybreaks

\renewcommand{\epsilon}{\varepsilon}
\renewcommand{\emptyset}{\varnothing}
\renewcommand{\geq}{\geqslant}
\renewcommand{\leq}{\leqslant}

\DeclareMathOperator{\dist}{dist}
\DeclareMathOperator{\tw}{tw}

\DeclareMathOperator{\simval}{simval}

\DeclareMathOperator{\pw}{pw}

\newcommand{\RR}{\mathbb{R}}

\newcommand{\PP}{\mathcal{P}}

\newcommand{\GG}{\mathcal{G}}
\newcommand{\HH}{\mathcal{H}}

\newcommand{\NN}{\mathbb{N}}

\renewcommand{\thefootnote}{\fnsymbol{footnote}}
\theoremstyle{plain}
\newtheorem{thm}{Theorem}
\newtheorem{lem}[thm]{Lemma}
\newtheorem{cor}[thm]{Corollary}

\newtheorem{prop}[thm]{Proposition}

\newtheorem{obs}[thm]{Observation}

\crefname{obs}{Observation}{Observations}
\newtheorem*{lem*}{Lemma}
\theoremstyle{definition}

\newtheorem*{conj*}{Conjecture}

\presetkeys%
{todonotes}%
{inline}{}
\date{}

\begin{document}

\title{\bf\fontsize{18pt}{18pt}\selectfont Graphs that are Quasi-Isometric to Graphs with Bounded Treewidth}

\author{Robert Hickingbotham\,\footnotemark[1]	}

\footnotetext[1]{Univ. Lyon, ENS de Lyon, UCBL, CNRS, LIP, France (\texttt{rd.hickingbotham@gmail.com}).}

\maketitle
\begin{abstract}
  In this paper, we characterise graphs that are quasi-isometric to graphs with bounded treewidth. Specifically, we prove that a graph is quasi-isometric to a graph with bounded treewidth if and only if it has a tree-decomposition where each bag consists of a bounded number of balls of bounded diameter. This result extends a characterisation by Berger and Seymour (2024) of graphs that are quasi-isometric to trees. Additionally, we characterise graphs that are quasi-isometric to graphs with bounded pathwidth and graphs that are quasi-isometric to graphs with bounded linewidth. As an application of these results, we show that graphs with bounded rank-width, graphs with bounded tree independence number, and graphs with bounded sim-width are quasi-isometric to graphs with bounded treewidth.
\end{abstract}

\renewcommand{\thefootnote}{\arabic{footnote}}
\section{Introduction}

Coarse graph theory is an emerging field that explores the global structure of graphs through the lens of Gromov's coarse geometry.\footnote{All graphs are simple, undirected, unweighted and finite unless stated otherwise.} Initiated by \citet{georgakopoulos2023graph}, this area presents numerous open questions, since many classical results from structural graph theory have a potential coarse analogue. Central to this topic is the notion of quasi-isometry, a generalisation of bi-Lipschitz maps that preserves the large-scale geometry of a metric space. A key objective is to identify conditions under which a complex graph is quasi-isometric to a simple graph.

In this paper, we characterise graphs that are quasi-isometric to graphs with bounded treewidth. Treewidth is a fundamental parameter in algorithmic and structural graph theory and is the standard measure of how similar a graph is to a tree. To formally state our results, we need the following definitions.

Let $G$ be a graph. For vertices $u,v\in V(G)$, a \defn{$(u,v)$-path} is a path in $G$ with end-points $u$ and $v$. The \defn{distance} $\dist_G(u,v)$ between $u$ and $v$ in $G$ is the length of the shortest $(u,v)$-path, or infinite if no such path exists. For $q \in \NN$, a \defn{$q$-quasi-isometry} of $G$ into a graph $H$ is a map $\phi \colon V(G) \to V(H)$ such that, for every $u,v\in V(G)$,
\begin{equation*}
    q^{-1} \cdot \dist_G(u,v) - q \leq \dist_{H}(\phi(u), \phi(v))\leq q \cdot \dist_G(u,v) + q,
\end{equation*}
and, for every $x \in V(H)$, there exists a vertex $v \in V(G)$ such that $\dist_{H}(x,\phi(v)) \leq q$. If such a map exists, then we say that $G$ is \defn{$q$-quasi-isometric} to $H$. We say a class of graphs $\GG$ is \defn{quasi-isometric} to a class of graphs $\mathcal{H}$ if there exists $q\in \NN$ such that every graph in $\GG$ is $q$-quasi-isometric to a graph in $\HH$.

A \defn{$T$-decomposition} of a graph $G$ is a pair ${(T,\beta)}$ where $T$ is a tree such that:
 \begin{itemize}
     \item ${\beta \colon V(T) \to 2^{V(G)}}$ is a function;
     \item for every edge ${uv \in E(G)}$, there exists a node ${x \in V(T)}$ with ${u,v \in \beta(x)}$; and 
     \item for every vertex ${v \in V(G)}$, the set $\{ x \in V(T) \colon v \in \beta(x) \}$ induces a non-empty connected subtree of~$T$. 
 \end{itemize}
We call $\beta(t)$ a \defn{bag} of the $T$-decomposition. The \defn{width} of ${(T,\beta)}$ is ${\max\{ \lvert \beta(x)\rvert \colon x \in V(T) \}-1}$. The \defn{treewidth}~$\tw(G)$ of $G$ is the minimum width of a $T$-decomposition of~$G$ for any tree $T$. The \defn{pathwidth}~$\pw(G)$ of $G$ is the minimum width of a $P$-decomposition of~$G$ for any path $P$.

For a set $X\subseteq V(G)$, the \defn{weak diameter} of $X$ is the maximum distance (in $G$) between any pair of vertices in $X$. For $d,k\in \NN$, a set $S\subseteq V(G)$ is \defn{$(k,d)$-centred} if $S$ has a partition $S_1,\dots, S_j$ into $j\leq k$ sets such that each $S_i$ has weak diameter at most $d$. A tree-decomposition of $G$ is \defn{$(k,d)$-centred} if each bag of the tree-decomposition is $(k,d)$-centred. 

Our main result states that $(k,d)$-centred tree-decompositions characterise graphs that are quasi-isometric to graphs with bounded treewidth.

\begin{thm}\label{CharacterisationThmTW}
    There is a function $q\colon \NN^2 \to \NN$ such that, for every $c,d,k\in \NN$:
    \begin{itemize}
        \item Every graph that has a $(k,d)$-centred tree-decomposition is $q(k,d)$-quasi-isometric to a graph with treewidth at most $2k-1$. 
        \item Every graph that is $c$-quasi-isometric to a graph with treewidth at most $k$ has a $(k+1,3c^2)$-centred tree-decomposition.
    \end{itemize}
\end{thm}

This result generalises the characterisation of graphs quasi-isometric to trees by \citet{BergerSeymour2024}, who showed that a graph is quasi-isometric to a tree if and only if it is connected and has a tree-decomposition where each bag has bounded weak diameter.

In addition to \cref{CharacterisationThmTW}, we prove an analogous result for graphs that are quasi-isometric to graphs with bounded pathwidth.

\begin{thm}\label{CharacterisationThmPW}
    There is a function $q\colon \NN^2 \to \NN$ such that, for every $c,d,k\in \NN$:
    \begin{itemize}
        \item Every graph that has a $(k,d)$-centred path-decomposition is $q(k,d)$-quasi-isometric to a graph with pathwidth at most $2k-1$.
        \item Every graph that is $c$-quasi-isometric to a graph with pathwidth at most $k$ has a $(k+1,3c^2)$-centred path-decomposition.
    \end{itemize} 
\end{thm}

We also extend the above results to infinite graphs with bounded linewidth. Linewidth is a new parameter introduced by \citet{NSSpathwidth2025} that generalises pathwidth to infinite graphs. A \defn{line} is a set that is linearly ordered by some relation $<$. For a line $L$, an \defn{$L$-decomposition} of a (possibly infinite) graph is a pair $(L,\beta)$ which satisfies the same conditions as in the definition of a $T$-decomposition. The \defn{width} of $(L,\beta)$ is ${\max\{ \lvert \beta(x)\rvert \colon x \in L \}-1}$. The \defn{linewidth} of a (possibly infinite) graph $G$ is the minimum integer $k$ such that $G$ admits an $L$-decomposition of width at most $k$ for some line $L$; if no such decomposition exists, then it is infinite. An $L$-decomposition is \defn{$(k,d)$-centred} if each bag in the decomposition is $(k,d)$-centred. 

For finite graphs, \citet{NSSpathwidth2025} observed that linewidth coincides with pathwidth. However, for infinite graphs, these parameters diverge. For example, the disjoint union of infinitely many one-way infinite paths has infinite pathwidth but linewidth equal to $1$. So linewidth is a more suitable parameter than pathwidth for infinite graphs. Indeed, \citet{CNSSsubtrees2025} showed that a graph has bounded linewidth if and only if all of its finite induced subgraphs have bounded pathwidth. 

In \cref{SectionLineWidth}, we prove the following characterisation for graphs that are quasi-isometric to graphs with bounded linewidth. Graphs in the following theorem statement may be infinite

\begin{thm}\label{CharacterisationThmLW}
    There is a function $q\colon \NN \times \NN \to \NN$ such that, for every $c,d,k\in \NN$:
    \begin{itemize}
        \item Every graph that has a $(k,d)$-centred line-decomposition is $q(k,d)$-quasi-isometric to a graph with linewidth at most $2k-1$.
        \item Every graph that is $c$-quasi-isometric to a graph with linewidth at most $k$ has a $(k+1,3c^2)$-centred line-decomposition.
    \end{itemize} 
\end{thm}

Our motivation for establishing the above characterisations is to make quasi-isometry more tractable to work with. To demonstrate this, we use \cref{CharacterisationThmTW} to show that several well-studied graph classes are quasi-isometric to graphs with bounded treewidth. 

\begin{thm}\label{ClassesQuasi}
    The following graph classes are quasi-isometric to graphs of bounded treewidth:
    \begin{itemize}
        \item Graphs with bounded rank-width;
        \item Graphs with bounded tree independence number;
        \item Graphs with bounded mim-width; and
        \item Graphs with bounded sim-width.
    \end{itemize}
\end{thm}

\citet{Gromov1993} introduced asymptotic dimension as a quasi-isometry invariant of metric spaces; see \cite{BellDranishnikov2008} for a survey. 

We now formally define asymptotic dimension. Let $G$ be a (possibly infinite) graph. For $\delta\geq 0$, a collection $\mathcal{U}$ of subsets of $V(G)$ is \defn{$\delta$-disjoint} if $\dist_G(X,Y)>\delta$ for all distinct $X,Y\in \mathcal{U}$. More generally, we say that it is \defn{$(k,\delta)$-disjoint} if $\mathcal{U}=\bigcup(\mathcal{U}\colon i\in [k])$ where for each $i\in [k]$, $\mathcal{U}_i$ is a $\delta$-disjoint subcollection of $\mathcal{U}$. A function $f\colon \RR^+ \to \RR^+$ is a \defn{$d$-dimensional control function} for $G$ if, for every $\delta \geq 0$, there is a $(d+1,\delta)$-disjoint partition $\mathcal{U}$ of $V(G)$ into sets with weak-diameter at most $f(\delta)$ in $G$. The \defn{asymptotic dimension} of a class of (possibly infinite) graphs $\GG$ is the minimum $d$ for which there exists a $d$-dimensional control function $f\colon \RR^+ \to \RR^+$ for all graphs $G\in \GG$, or infinite if no such $d$ exists. 

\citet[Theorem~1.2]{bonamy2023asymptotic} showed that every graph class with bounded treewidth has asymptotic dimension at most $1$. Therefore, each of the graph classes in \cref{ClassesQuasi} have asymptotic dimension at most $1$.

\begin{cor}\label{ClassesAsymptotic}
    The following graph classes have asymptotic dimension at most $1$:
    \begin{itemize}
        \item Graphs with bounded rank-width;
        \item Graphs with bounded tree independence number;
        \item Graphs with bounded mim-width; and
        \item Graphs with bounded sim-width.
    \end{itemize}
\end{cor}

\citet{Benjamini2012} previously studied vertex-transitive graphs that are quasi-isometric to graphs with bounded treewidth in the context of separation profile. \citet{dragan2014collective} studied $(k,r)$-centred tree-decomposition in the context of designing collective additive tree spanners in graphs.

\textbf{Note:}
While preparing this manuscript, \citet{NSS2025treewidth} announced an independent proof of \cref{CharacterisationThmTW,CharacterisationThmPW} where they achieve tight bounds on the width. The strength of our proof lies in its conceptual simplicity. This is good news for coarse graph theory since most proofs in this area are long and technical. Moreover, our results extend to infinite graphs with bounded linewidth where it is unclear whether that is the case for their results.

\section{Treewidth and Pathwidth}
In this section, we prove \cref{CharacterisationThmTW,CharacterisationThmPW}. See \citet{diestel2017graphtheory} for undefined terms and notation. Our first lemma shows that any graph quasi-isometric to a graph with bounded treewidth/pathwidth/linewidth has a $(k,d)$-centred tree/path/line-decomposition.

\begin{lem}\label{Quasi2kd}
   For every $c,k\in \NN$ and every tree or line $T$, if a (possibly infinite) graph $G$ is $c$-quasi-isometric to a graph $H$ where $H$ has a $T$-decomposition with width at most $k$, then $G$ has a $(k+1,3c^2)$-centred $T$-decomposition.
\end{lem}

\begin{proof}
    Let $\phi\colon V(G)\to V(H)$ be a $c$-quasi-isometry from $G$ to $H$. Let $(T,\beta')$ be a $T$-decomposition of $H$ with width at most $k$. For each $x\in V(H)$, define 
    $$A_x=\{v\in V(G)\colon \dist_{H}(\phi(v),x)\leq c\}.$$ 
    Then, for all $u,v\in A_x$, we have 
    $$\dist_G(u,v)\leq c\dist_{H}(\phi(u),\phi(v))+c^2\leq c(\dist_{H}(\phi(u),x)+\dist_{H}(\phi(v),x))+c^2\leq 3c^2.$$ 
    For each node $t\in V(T)$, define $\beta(t)=\bigcup (A_x\colon x\in \beta'(t))$. Then $\beta(t)$ is the union of $|\beta'(t)|\leq k+1$ sets where each set has weak diameter at most $3c^2$. 
    
    For each vertex $v\in V(G)$, let $C_v=\{x\in V(H)\colon \dist_{H}(\phi(v),x)\leq c\}$. Then, for each node $t\in V(T)$, $v\in \beta(t)$ if and only if $C_v\cap \beta'(t)\neq \emptyset$. Since $C_v$ induces a non-empty and connected subgraph of $H$, it follows that $\{t\in V(T)\colon v\in \beta(t)\}$ induces a non-empty and connected subtree of $T$ if $T$ is a tree, or a non-empty interval if $T$ is a line. 
    
    Finally, consider an edge $uv\in E(G)$. By the definition of quasi-isometry,
    $$\dist_{H}(\phi(u),\phi(v))\leq c\dist_G(u,v)+c\leq 2c.$$
    Then $u,v\in A_x$ where $x\in V(H)$ is the middle vertex on a shortest $(\phi(u),\phi(v))$-path in $H$. This means that $u,v\in \beta(t)$ whenever $x\in \beta'(t)$. Therefore, $(T,\beta)$ is a $(k+1,3c^2)$-centred $T$-decomposition of $G$.
\end{proof}

Let $G$ be a graph and $(T,\beta)$ be a tree-decomposition of $G$. An \defn{independent set} in $G$ is a set of pairwise non-adjacent vertices. A set $D\subseteq V(G)$ is \defn{dominating} if every vertex in $V(G)\setminus D$ is adjacent to a vertex in $D$. The \defn{independence number} $\alpha(G)$ of $G$ is the maximum size of an independent set in $G$. The \defn{domination number} $\gamma(G)$ of $G$ is the minimum size of a dominating set in $G$. Clearly $\gamma(G)\leq \alpha(G)$. The \defn{independence number} of $(T,\beta)$ is $\max\{\alpha(G[\beta(t)])\colon t\in V(T)\}$ and the \defn{domination number} of $(T,\beta)$ is $\max\{\gamma(G[\beta(t)])\colon t\in V(T)\}$. The \defn{tree/path independence number} of $G$ is the minimum independence number of a tree/path-decomposition of $G$.  

We now outline our proof that every graph with a $(k,d)$-centred tree-decomposition is quasi-isometric to a graph with bounded treewidth. The proof proceeds in two steps. First, we augment the graph by adding edges to vertices that share a common bag and are within distance $d$ of each other. This results in a graph with tree independence number at most $k$. Next, we use a lemma of \citet{DN2022Asymdim} to partition the graph into connected parts of bounded weak diameter which we then contract to make the graph bipartite. The key observation here is that tree independence number is preserved under contractions. Consequentially, the resulting bipartite graph has a tree-decomposition where each bag contains at most $2k$ vertices. Since both operations are preserved under quasi-isometry, this completes the proof. The remainder of this section formalises this argument.
    
\begin{lem}\label{DecompToInd}
    For all $d,k\in \NN$ and every tree $T$, every graph $G$ that has a $(k,d)$-centred $T$-decomposition is $d$-quasi-isometric to a graph $H$ that has a $T$-decomposition with independence number at most $k$. 
\end{lem}

\begin{proof}
    We may assume that $G$ is connected. Let $(T,\beta)$ be a $(k,d)$-centred $T$-decomposition of $G$. Construct $H$ from $G$ by adding the edge $uv$ to $G$ whenever $u,v\in \beta(t)$ for some $t\in V(T)$ and $\dist_G(u,v)\leq d$. Then $(T,\beta)$ is also a $T$-decomposition of $H$. Moreover, since each bag in the $T$-decomposition of $H$ is the union of at most $k$ cliques, it follows that the independence number of this $T$-decomposition of $H$ is at most $k$.

    We now show that the identity map $\phi\colon V(G)\to V(H)$, defined by $\phi(v)=v$ for all $v\in V(G)$, is a $d$-quasi-isometry from $G$ to $H$. Let $u,v\in V(H)$. Then $\dist_{H}(u,v)\leq \dist_G(u,v)$ since every path in $G$ is a path in $H$. Let $(u=w_0,w_1,w_2,\dots,w_{n-1},w_n=v)$ be a $(u,v)$-path in $H$ of length $n=\dist_{H}(u,v)$. For each $i\in \{1,\dots,n\}$, $G$ contains a $(w_{i-1},w_{i})$-path $P_i$ of length at most $d$. Concatenating the paths $P_1,\dots,P_n$ gives a $(u,v)$-walk in $G$ of length at most $dn$. Thus $\dist_G(u,v)\leq d\cdot \dist_{H}(u,v)$, as required.
\end{proof}

 Let $G$ be a graph. A \defn{partition} of $G$ is a collection $\PP$ of non-empty sets of vertices in $G$ such that each vertex of $G$ is in exactly one element of $\PP$ and each element of $\PP$ induces a connected subgraph of $G$. An element of $\PP$ is called a \defn{part}. The \defn{quotient} of $\PP$ (with respect to $G$) is the graph \defn{$G/\PP$} whose vertices are parts of $\PP$, where distinct parts $A,B\in \PP$ are adjacent in $G/\PP$ if and only if some vertex in $A$ is adjacent in $G$ to some vertex in $B$. Note that $G/\PP$ is isomorphic to the graph obtained from $G$ by contracting each part of $\PP$ into a single vertex.
 
 The following observation connects partitions to quasi-isometry.

\begin{obs}\label{QuasiIsometryObservation}
    For every $d\in \NN$, for every graph $G$, for every partition $\PP$ of $G$ where each part has weak diameter less than $d$, $G$ is $d$-quasi-isometric to $G/\PP$.
\end{obs}

The next result is due to \citet[Proposition~3.9]{DALLARD2024indep}.\footnote{Note that their proposition is stated in terms of tree-decompositions and induced minors. It is easy to see that the proof of their proposition implies this version of it.}

\begin{prop}[\cite{DALLARD2024indep}]\label{IndependentMinor}
    For every $k\in \NN$ and every tree $T$, for every graph $G$ that has a $T$-decomposition with independence number at most $k$, for every partition $\PP$ of $G$, the graph $G/\PP$ has a $T$-decomposition with independence number at most $k$. 
\end{prop}

The next lemma is due to \citet[Lemma 10]{DN2022Asymdim}. 

\begin{lem}[\cite{DN2022Asymdim}]\label{WeakDiamColouring}
    There is a function $f\colon \NN \to \NN$ such that, for every $k\in \NN$, every graph $G$ that has a tree-decomposition with domination number at most $k$ has a partition $\PP$ of $G$ such that each part has weak diameter less than $f(k)$ and $G/\PP$ is bipartite.
\end{lem}

We now show that graphs with bounded tree independence number are quasi-isometric to graphs with bounded treewidth.

\begin{lem}\label{IndToTW}
    There is a function $f\colon \NN \to \NN$ such that, for every $k\in \NN$ and tree $T$, every graph $G$ that has a $T$-decomposition with independence number at most $k$ is $f(k)$-quasi-isometric to a graph $H$ that has a $T$-decomposition with width less than $2k-1$. 
\end{lem}

\begin{proof}
    Let $f\colon \NN \to \NN$ be given by \cref{WeakDiamColouring}. Let $(T,\beta)$ be a $T$-decomposition of $G$ with independence number at most $k$. Then $(T,\beta)$ has domination number at most $k$. By \cref{WeakDiamColouring}, $G$ has a partition $\PP$ where each part has weak diameter at most $f(k)$ and $G/\PP$ is bipartite. Let $H=G/\PP$. By \cref{QuasiIsometryObservation}, $G$ is $f(k)$-quasi-isometric to $H$. By \cref{IndependentMinor}, $H$ has a $T$-decomposition $(T,\beta')$ with independence number at most $k$. Since $H$ is bipartite, it follows that $|\beta'(t)|\leq 2k$ for every node $t\in V(T)$.
\end{proof}

The next observation is folklore. 

\begin{obs}\label{QuasiQuasi}
    For all $c,q\in \NN$, if a graph $G$ is a $c$-quasi-isometric to a graph $G_1$ and if $G_1$ is $q$-quasi-isometric to a graph $G_2$, then $G$ is $q(c+2)$-quasi-isometric to $G_2$.
\end{obs}

We now prove our main lemma.
\begin{lem}\label{MainLemma}
    There is a function $q\colon \NN^2 \to \NN$ such that, for all $d,k\in \NN$ and every tree $T$, every graph $G$ that has a $(k,d)$-centred $T$-decomposition is $q(k,d)$-quasi-isometric to a graph that has a $T$-decomposition with width at most $2k-1$. 
\end{lem}

\begin{proof}
    Let $f\colon \NN \to \NN$ be the function given by \cref{WeakDiamColouring}. By \cref{DecompToInd}, $G$ is $d$-quasi-isometric to a graph $G_1$ that has a $T$-decomposition with independence number $k$. By \cref{IndToTW}, $G_1$ is $f(k)$-quasi-isometric to a graph $G_2$ that has a $T$-decomposition with width at most $2k-1$. The claim then follows from \cref{QuasiQuasi} by setting $q(k,d)=(d+2) f(k)$.
\end{proof}

\cref{CharacterisationThmTW,CharacterisationThmPW} immediately follows from \cref{Quasi2kd,MainLemma}.

\section{Application}

In this section, we use \cref{CharacterisationThmTW} to show that several graph classes are quasi-isometric to graphs with bounded treewidth. To unify our approach, we only show that graphs with bounded sim-width are quasi-isometric to graphs with bounded treewidth. We choose the parameter sim-width because it encompasses many graph classes. Specifically, \citet{Vatshelle2012} proved that graphs with bounded rank-width have bounded mim-width; \citet{KANG2017} showed that graphs with bounded mim-width have bounded sim-width; and \citet{MUNARO2023} showed that graphs with bounded tree independence number have bounded sim-width. Consequently, proving \cref{ClassesQuasi} reduces to the following.

\begin{thm}\label{SimWidthQuasi}
    There is a function $q\colon \NN \to \NN$ such that, for every $k\in \NN$, every graph $G$ with sim-width at most $k$ is $q(k)$-quasi-isometric to a graph with treewidth at most $12k-1$.
\end{thm}

Since we will not discuss rank-width and mim-width any further, we omit their definitions. See \cite{OS06} for the definition of rank-width and \cite{Vatshelle2012} for the definition of mim-width. 

We now define sim-width. For a graph $G$, let \defn{$\simval_G$}$\colon 2^{V(G)}\to \NN$ be the function such that, for $A\subseteq V(G)$, $\simval_G(A)$ is the maximum size of an induced matching $\{a_1b_1, a_2b_2,\dots, a_mb_m\}$ in $G$ where $a_1,\dots,a_m\in A$ and $b_1\dots,b_m\in V(G)\setminus A$. For a graph $G$, a pair $(T, L)$ of a subcubic tree $T$ and a function $L$ from $V(G)$ to the set of leaves of $T$ is a \defn{branch-decomposition}. For each edge $e$ of $T$, let $(A_1^e,A_2^e)$ be the vertex partition of $G$ associated with $e$, where $T_1^e$ and $T_2^e$ are the two components of $T-e$ and, for each $i\in \{1,2\}$, $A_i^e$ is the set of vertices in $G$ mapped to the leaves in $T_i^e$. For a branch-decomposition $(T, L)$ of a graph $G$ and an edge $e$ in $T$, the \defn{width} of $e$ is $\simval_G(A_e^1)$ where $(A_e^1,A_e^2)$ is the vertex partition associated with $e$. The \defn{width} of $(T, L)$ is the maximum width over all edges in $T$. The \defn{sim-width} of $G$ is the minimum width of a branch-decomposition of $G$.

\begin{lem}\label{LemmaSimWidth}
    For every $k\in \NN$, every graph $G$ with sim-width at most $k$ has a tree-decomposition with domination number at most $6k$.
\end{lem}

\begin{proof}
    We may assume that $G$ is connected. If $|V(G)|\leq 1$, then we are done by placing $V(G)$ into a single bag. So we may assume that every vertex in $G$ is incident to an edge. Let $(T,L)$ be a branch-decomposition of $G$ with width at most $k$. For each edge $uv\in E(G)$, for each node $t$ on the unique $(L(u),L(v))$-path in $T$, add $u$ and $v$ to the bag $\beta(t)$. We claim that $(T,\beta)$ is a $T$-decomposition of $G$ with domination number at most $6k$. 
    
    By construction, for each edge $uv\in E(G)$, we have $u,v\in \beta(L(u))$. Moreover, for each vertex $v\in V(G)$, the graph $T[\{t\in V(T)\colon v\in \beta(t)\}]$ is the union of paths in $T$ that each contain $L(v)$ as an end-point. So $T[\{t\in V(T)\colon v\in \beta(t)\}]$ is non-empty and connected. Therefore, $(T,\beta)$ is a tree-decomposition of $G$. 

    For each leaf $\ell$ in $T$, every vertex in $\beta(\ell)$ is adjacent in $G$ to $L^{-1}(\ell)$ or is itself $L^{-1}(\ell)$. So $\beta(\ell)$ has domination number $1$. Let $t\in V(T)$ be a non-leaf node. Let $(A_1^t,A_2^t,A_3^t)$ be the partition of $\beta(t)$ where $T_1^t,T_2^t,T_3^t$ are the components of $T-t$, and for each $i\in \{1,2,3\}$, $A_i^t$ is the set of vertices in $G$ mapped to a leaf in $T_i^t$ that has a neighbour mapped to a different component of $T-t$. Let $\{a_1 b_1,\dots,a_m b_m\}$ be a maximum induced matching in $G$ where $a_1,\dots a_m\in A_1^t$ and $b_1,\dots,b_m\in A_2^t\cup A_3^t$. Then $m\leq k$. Since every vertex in $A_1^t$ is adjacent to a vertex in $\{a_1,b_1,\dots,a_m,b_m\}$, it follows that $G[A_1^t]$ has domination number at most $2k$. By symmetry, $G[A_2^t]$ and $G[A_3^t]$ also have domination number at most $2k$. So $G[\beta(t)]$ has domination number at most $6k$, as required.
\end{proof}

\cref{SimWidthQuasi} follows from \cref{CharacterisationThmTW} and \cref{LemmaSimWidth} since if a graph induced by a set of vertices $X$ has domination number at most $k$, then the set $X$ is $(k,3)$-centred.

\section{Linewidth}\label{SectionLineWidth}

In this section, we characterise graphs that are quasi-isometric to graphs with bounded linewidth. Throughout this section, graphs may be infinite. Our proof is essentially the same as what we did for treewidth and pathwidth, except we replace \cref{WeakDiamColouring} by \cref{LineWidth2Colouring}. The \defn{line independence number} of a graph $G$ is the minimum $k\in \NN$ such that $G$ admits an $L$-decomposition $(L,\beta)$ for some line $L$ where each bag has independence number at most $k$; if no such decomposition exists, then it is infinite. 

\begin{lem}\label{LineWidth2Colouring}
    There exists a function $f \colon \NN \to \NN$ such that, for every $k\in \NN$, every graph $G$ with line independence number at most $k$ has a partition $\PP$ such that each part has weak diameter less than $f(k)$ and $G/\PP$ is bipartite.
\end{lem}

To prove \cref{LineWidth2Colouring}, we require the following compactness result of \citet[Theorem~A.2]{bonamy2023asymptotic}.

\begin{lem}[\cite{bonamy2023asymptotic}]\label{Compact}
    Let $G$ be a graph. If the class of its finite induced subgraphs has asymptotic dimension at most $n$, then $G$ itself has asymptotic dimension at most $n$. 
\end{lem}

\begin{proof}[Proof of \cref{LineWidth2Colouring}]
    Fix $k\in \NN$, and let $(L,\beta)$ be an $L$-decomposition of $G$ for some line $L$, where each bag $\beta(x)$ has independence number at most $k$. It suffices to show that there is a constant $c_k\in \NN$ such that $G$ admits a $2$-colouring in which each monochromatic component has weak diameter less than $c_k$.
    
    Let $\GG$ denote the class of the finite induced subgraphs of $G$. For any graph $G'\in \GG$, the family $(\beta'(x)=\beta(x)\cap V(G')\colon x\in L)$ defines an $L$-decomposition of $G'$ where each bag has independence number at most $k$. Thus $G'$ has tree independence number at most $k$. By \cref{ClassesAsymptotic}, the graph class $\GG$ has asymptotic dimension at most $1$. Applying \cref{Compact}, we conclude that $G$ also has asymptotic dimension at most $1$. It follows from the definition of asymptotic dimension that there exists a constant $c_k\in \NN$ such that $G$ has a $2$-colouring in which each monochromatic component of $G$ has weak diameter less than $c_k$, as required.
\end{proof}

\cref{CharacterisationThmLW} now follows from the next lemma together with \cref{Quasi2kd}. The proof for this lemma is identical to that of \cref{MainLemma}, except that we use \cref{LineWidth2Colouring} in place of \cref{WeakDiamColouring}. We omit the details.

\begin{lem}\label{MainLemmaLINE}
    There exists a function $q\colon \NN^2 \to \NN$ such that, for all $d,k\in \NN$ and line $L$, every graph $G$ that has a $(k,d)$-centred $L$-decomposition is $q(k,d)$-quasi-isometric to a graph that has an $L$-decomposition with width at most $2k-1$. 
\end{lem}

\subsubsection*{Acknowledgement}
This work was initiated at the 2024 Barbados Graph Theory Workshop held at Bellairs Research Institute in March 2024. Thanks to the organisers and participants for providing a stimulating research environment. Thanks to James Davies for pointing out \cite{NSS2025treewidth}, to Paul Seymour for asking about linewidth, and to Raj Kaul for pointing out a typo in this paper.

{
\fontsize{11pt}{12pt}
\selectfont
	
\hypersetup{linkcolor={red!70!black}}
\setlength{\parskip}{2pt plus 0.3ex minus 0.3ex}

\bibliographystyle{DavidNatbibStyle}
\bibliography{main.bbl}

\begin{thebibliography}{17}
\providecommand{\natexlab}[1]{#1}
\providecommand{\msn}[1]{MR:\,\href{http://www.ams.org/mathscinet-getitem?mr=MR{#1}}{#1}}
\providecommand{\ZBL}[1]{Zbl:\,\href{https://www.zentralblatt-math.org/zmath/en/search/?q=an:#1}{#1}}
\providecommand{\url}[1]{\texttt{#1}}
\providecommand{\urlprefix}{}
\expandafter\ifx\csname urlstyle\endcsname\relax
  \providecommand{\doi}[1]{doi:\discretionary{}{}{}#1}\else
  \providecommand{\doi}{doi:\discretionary{}{}{}\begingroup
  \urlstyle{rm}\Url}\fi

\bibitem[{Bell and Dranishnikov(2008)}]{BellDranishnikov2008}
\textsc{G.~Bell and A.~Dranishnikov}.
\newblock Asymptotic dimension.
\newblock \emph{Topology and its Applications}, 155:1265--1296, 2008.

\bibitem[{Benjamini et~al.(2012)Benjamini, Schramm, and
  Tim{\'a}r}]{Benjamini2012}
\textsc{Itai Benjamini, Oded Schramm, and {\'A}d{\'a}m Tim{\'a}r}.
\newblock \href{https://doi.org/10.4171/GGD/168}{On the separation profile of
  infinite graphs}.
\newblock \emph{Groups, Geometry, and Dynamics}, 6(4):639--658, 2012.

\bibitem[{Berger and Seymour(2024)}]{BergerSeymour2024}
\textsc{Eli Berger and Paul Seymour}.
\newblock \href{https://doi.org/10.1007/s00493-024-00088-1}{Bounded-diameter
  tree-decompositions}.
\newblock \emph{Combinatorica}, 44(1):659--674, 2024.

\bibitem[{Bonamy et~al.(2023)Bonamy, Bousquet, Esperet, Groenland, Liu, Pirot,
  and Scott}]{bonamy2023asymptotic}
\textsc{Marthe Bonamy, Nicolas Bousquet, Louis Esperet, Carla Groenland,
  Chun-Hung Liu, Fran{\c{c}}ois Pirot, and Alexander Scott}.
\newblock \href{https://doi.org/10.4171/JEMS/1341}{Asymptotic dimension of
  minor-closed families and {A}ssouad--{N}agata dimension of surfaces}.
\newblock \emph{Journal of the European Mathematical Society},
  26(10):3739--3791, 2023.

\bibitem[{Chudnovsky et~al.(2025)Chudnovsky, Nguyen, Scott, and
  Seymour}]{CNSSsubtrees2025}
\textsc{Maria Chudnovsky, Tung Nguyen, Alex Scott, and Paul Seymour}.
\newblock The vertex sets of subtrees of a tree.
\newblock 2025.

\bibitem[{Dallard et~al.(2024)Dallard, Milanič, and
  Štorgel}]{DALLARD2024indep}
\textsc{Clément Dallard, Martin Milanič, and Kenny Štorgel}.
\newblock
  \href{https://www.sciencedirect.com/science/article/pii/S0095895623000886}{Treewidth
  versus clique number. {II}. tree-independence number}.
\newblock \emph{Journal of Combinatorial Theory, Series B}, 164:404--442, 2024.

\bibitem[{Diestel(2017)}]{diestel2017graphtheory}
\textsc{Reinhard Diestel}.
\newblock \href{https://doi.org/10.1007/978-3-662-53622-3}{Graph theory}, vol.
  173 of \emph{Graduate Texts in Mathematics}.
\newblock Springer, 5th edn., 2017.

\bibitem[{Dragan and Abu-Ata(2014)}]{dragan2014collective}
\textsc{Fedor~F. Dragan and Malik Abu-Ata}.
\newblock \href{https://doi.org/10.1016/j.tcs.2014.06.007}{Collective additive
  tree spanners of bounded tree-breadth graphs with generalizations and
  consequences}.
\newblock \emph{Theoretical Computer Science}, 547:1--17, 2014.

\bibitem[{Dvo{\v{r}}{\'a}k and Norin(2025)}]{DN2022Asymdim}
\textsc{Zden{\v{e}}k Dvo{\v{r}}{\'a}k and Sergey Norin}.
\newblock \href{https://doi.org/10.1016/j.ejc.2022.103631}{Asymptotic dimension
  of intersection graphs}.
\newblock \emph{European J. Combin.}, 123:103631, 2025.

\bibitem[{Georgakopoulos and Papasoglu(2023)}]{georgakopoulos2023graph}
\textsc{Agelos Georgakopoulos and Panos Papasoglu}.
\newblock \href{http://arxiv.org/abs/2305.07456}{Graph minors and metric
  spaces}.
\newblock 2023, arXiv:2305.07456.

\bibitem[{Gromov(1993)}]{Gromov1993}
\textsc{M.~Gromov}.
\newblock Asymptotic invariants of infinite groups.
\newblock In \emph{Geometric Group Theory: Proceedings of the Symposium held in
  Sussex, 1991}, vol.~2. Cambridge University Press, 1993.

\bibitem[{Kang et~al.(2017)Kang, Kwon, Strømme, and Telle}]{KANG2017}
\textsc{Dong~Yeap Kang, O-joung Kwon, Torstein~J.F. Strømme, and Jan~Arne
  Telle}.
\newblock
  \href{https://www.sciencedirect.com/science/article/pii/S0304397517306618}{A
  width parameter useful for chordal and co-comparability graphs}.
\newblock \emph{Theoret. Comput. Sci.}, 704:1--17, 2017.

\bibitem[{Munaro and Yang(2023)}]{MUNARO2023}
\textsc{Andrea Munaro and Shizhou Yang}.
\newblock
  \href{https://www.sciencedirect.com/science/article/pii/S030439752300138X}{On
  algorithmic applications of sim-width and mim-width of
  ({$H\sb{1},H\sb{2}$})-free graphs}.
\newblock \emph{Theoret. Comput. Sci.}, 955:113825, 2023.

\bibitem[{Nguyen et~al.(2025{\natexlab{a}})Nguyen, Scott, and
  Seymour}]{NSS2025treewidth}
\textsc{Tung Nguyen, Alex Scott, and Paul Seymour}.
\newblock \href{http://arxiv.org/abs/2501.09839}{Coarse tree-width}.
\newblock 2025{\natexlab{a}}, arXiv:2501.09839.

\bibitem[{Nguyen et~al.(2025{\natexlab{b}})Nguyen, Scott, and
  Seymour}]{NSSpathwidth2025}
\textsc{Tung Nguyen, Alex Scott, and Paul Seymour}.
\newblock Path-width and additive quasi-isometry.
\newblock 2025{\natexlab{b}}.

\bibitem[{Oum and Seymour(2006)}]{OS06}
\textsc{Sang-il Oum and Paul Seymour}.
\newblock \href{https://doi.org/10.1016/j.jctb.2005.10.006}{Approximating
  clique-width and branch-width}.
\newblock \emph{J. Combin. Theory Ser. B}, 96(4):514--528, 2006.

\bibitem[{Vatshelle(2012)}]{Vatshelle2012}
\textsc{Martin Vatshelle}.
\newblock \href{https://hdl.handle.net/1956/6166}{New width parameters of
  graphs}.
\newblock Ph.D. thesis, University of Bergen, 2012.

\end{thebibliography}
}

\end{document}